\documentclass[12pt,a4paper]{article}
\usepackage{}

\usepackage{graphicx,tikz}
\usepackage{latexsym}
\usepackage{amsmath}
\usepackage{amssymb}
\usepackage{arydshln}

\usepackage{cite}
\usepackage{stmaryrd}
\usepackage{enumerate}

\usepackage{hyperref}

\usepackage{color}
\usepackage{lineno}
\usepackage{graphicx}
\usepackage{ae}
\usepackage{amsmath}
\usepackage{amssymb}
\usepackage{latexsym}
\usepackage{url}
\usepackage{epsfig}
\usepackage{mathrsfs}
\usepackage{amsfonts}
\usepackage{amsthm}
\usepackage{bm}

\newcommand{\Cay}{\mathrm{Cay}}
\newcommand{\Dih}{\mathrm{Dih}}
\newcommand{\Dic}{\mathrm{Dic}}
\newcommand{\SC}{\mathrm{SC}}

\newtheorem{theorem}{Theorem}[section]

\newtheorem{prop}[theorem]{Proposition}
\newtheorem{lemma}[theorem]{Lemma}

\newtheorem{cor}[theorem]{Corollary}

\newtheorem{example}{Example}

\theoremstyle{definition}
\newtheorem{definition}{Definition}

\numberwithin{equation}{section} 
\allowdisplaybreaks    

\def\qed{\hfill$\Box$\vspace{12pt}}

\long\def\delete#1{}

\usepackage{xcolor}
\usepackage[normalem]{ulem}

\begin{document}
\title{Fractional revival on semi-Cayley graphs over abelian groups}
\author{Jing Wang$^{a,b,c}$,~Ligong Wang$^{a,b,c}$$^,$\thanks{Supported by the National Natural Science Foundation of China (Nos. 11871398 and 12271439).}~,~Xiaogang Liu$^{a,b,c,}$\thanks{Supported by the Guangdong Basic and Applied
Basic Research Foundation (No. 2023A1515010986).}~$^,$\thanks{Corresponding author. Email addresses: wj66@mail.nwpu.edu.cn, lgwangmath@163.com, xiaogliu@nwpu.edu.cn}
\\[2mm]
{\small $^a$School of Mathematics and Statistics,}\\[-0.8ex]
{\small Northwestern Polytechnical University, Xi'an, Shaanxi 710072, P.R.~China}\\
{\small $^b$Research \& Development Institute of Northwestern Polytechnical University in Shenzhen,}\\[-0.8ex]
{\small Shenzhen, Guangdong 518063, P.R. China}\\
{\small $^c$Xi'an-Budapest Joint Research Center for Combinatorics,}\\[-0.8ex]
{\small Northwestern Polytechnical University, Xi'an, Shaanxi 710129, P.R. China}
}
\date{}

\openup 0.5\jot
\maketitle

\begin{abstract}
In this paper, we investigate the existence of fractional revival on semi-Cayley graphs over finite abelian groups. We give some necessary and sufficient conditions for semi-Cayley graphs over finite abelian groups admitting fractional revival. We also show that integrality is necessary for some semi-Cayley graphs admitting fractional revival. Moreover, we characterize the minimum time when semi-Cayley graphs admit fractional revival. As applications, we give examples of certain Cayley graphs over the generalized dihedral groups and generalized dicyclic groups admitting fractional revival.

\smallskip


\end{abstract}

\section{Introduction}

Let $\Gamma$ be a graph with vertex set $V(\Gamma)$ and edge set $E(\Gamma)$. The \emph{adjacency matrix} of $\Gamma$ is denoted by $A=A(\Gamma)=(a_{u,v})_{u,v\in V(\Gamma)}$, where $a_{u,v}=1$ if $uv\in E(\Gamma)$, and $a_{u,v}=0$ otherwise. The \emph{transition matrix} \cite{FarhiG98} of $\Gamma$ with respect to $A$ is defined by
$$
H(t) = \exp(\imath tA)=\sum_{k=0}^{\infty}\frac{\imath^{k} A^{k} t^{k}}{k!}, ~ t \in \mathbb{R},~\imath=\sqrt{-1}.
$$
where $\mathbb{R}$ is the set of real numbers.
Let $\mathbb{C}^n$ denote the $n$-dimensional vector space over complex field and $\mathbf{e}_u$ the standard basis vector in $\mathbb{C}^n$ indexed by the vertex $u$.
If $u$ and $v$ are distinct vertices in $\Gamma$ and there is a time $t$ such that
\begin{equation}\label{FRCAY-EQUATION3}
H(t)\mathbf{e}_u=\alpha\mathbf{e}_u+\beta\mathbf{e}_v,
\end{equation}
where $\alpha$ and $\beta$ are complex numbers and $|\alpha|^2+|\beta|^2=1$, then we say that $\Gamma$ admits \emph{fractional revival} from $u$ to $v$ at time $t$ or \emph{$(\alpha, \beta)$-revival} occurs from $u$ to $v$ at time $t$.
In particular, if $\alpha=0$, we say $\Gamma$ admits \emph{perfect state transfer} from $u$ to $v$ at time $t$; and if $\beta=0$, we say $G$ is \emph{periodic} at vertex $u$ at time $t$. We say \emph{$\Gamma$ has $(\alpha, \beta)$-revival} if there is a permutation matrix $Q$ of order two (with no fixed points) and one time $t$ such that
\begin{equation*}\label{ai+bq}
H(t)=\alpha I+\beta Q,
\end{equation*}
where $I$ is the identity matrix. That is, $\Gamma$ has $(\alpha, \beta)$-revival if and only if $(\alpha, \beta)$-revival occurs from every vertex of $\Gamma$ \cite{ChanCCTVZ20}.




Quantum state transfer in quantum networks first introduced by Bose in \cite{Bose03} is a very important research content for quantum communication protocols. Up until now, many wonderful results on perfect state transfer have been achieved. See \cite{Basic09, CaoCL20, CaoF21, chris1, chris2, Tan19, CC, Coh14, Coh19, Coutinho15, CoutinhoL2015, CGodsil, Godsil12, GodsilKSS12, LiLZZ21, LiuW2021} for example. However, to the best of our knowledge, only few results on fractional revival have been given, which are listed as follows:
\begin{itemize}
  \item A systematic study of fractional revival at two sites in XX quantum spin chains was given in \cite{GenestVZ16}.
  \item An example of a graph that admits balanced fractional revival between antipodes was given in \cite{BernardCLTV18}.
  \item A generalizations of fractional revival between cospectral and strongly cospectral vertices to arbitrary subsets of vertices, was given in \cite{BernardCLTV18}.
  \item A cycle admits fractional revival if and only if it has four or six vertices, and a path admits fractional revival if and only it has two, three or four vertices \cite{CCTVZ19}.
  \item Fractional revival in graphs, whose adjacency matrices belong to the Bose--Mesner algebra of association schemes, was given in \cite{ChanCCTVZ20}.
  \item An indication on how fractional revival can be swapped with perfect state transfer by modifying chains through isospectral deformations was given in \cite{SVZ22}.
\item A class of graphs admitting fractional revival between non-cospectral vertices was proposed in \cite{GodsilZ22}.
  \item A characterization on Cayley graphs over abelian groups admitting fractional revival was given in \cite{WangWliu22,Caoluo22}.
 \end{itemize}

Let $G$ be a group and $S$ a subset of $G$ with $S=S^{-1}=\left\{s^{-1}| s\in S\right\}$ (inverse-closed). The \emph{Cayley graph} $\Gamma=\Cay(G,S)$ is a graph whose vertex set is $G$ and edge set is $\{\{g,sg\}|s\in S\}$. It is known \cite{Caoluo22} that $(\alpha,\beta)$-revival occurs on $\Gamma=\Cay(G,S)$ from a vertex if and only if $(\alpha, \beta)$-revival occurs from every vertex of $\Gamma=\Cay(G,S)$.

As a generalization of Cayley graphs, de Resmini and Jungnickel \cite{ResJ92} introduced the definition of semi-Cayley graphs.

\begin{definition}
Let $G$ be a finite group with the identity $1$, and let $R$, $L$ and $S$ be subsets of $G$ such that $R$ and $L$ are inverse-closed subsets and $1\notin R,L$.
The \emph{semi-Cayley graph} $\SC(G,R,L,S)$ is an undirected graph with the vertex set $\{(g,0),(g,1)|g\in G\}$ and the edge set consisting of:
 \begin{align*}
 &\left\{\{(g,0),(h,0)\}|hg^{-1}\in R\right\},~~\text{(the~set~of~right~edges),}\\[0.2cm]
 &\left\{\{(g,1),(h,1)\}|hg^{-1}\in L\right\},~~\text{(the~set~of~left~edges),}\\[0.2cm]
 &\left\{\{(g,0),(h,1)\}|hg^{-1}\in S\right\},~~\text{(the~set~of~spoke~edges),}
 \end{align*}
where $\{(g,0)|g\in G\}$ and $\{(g,1)|g\in G\}$ are two orbits of $\SC(G,R,L,S)$.
\end{definition}

The purpose of this paper is to find the existence of fractional revival on semi-Cayley graphs. It is known \cite{AreT13} that every Cayley graph over a group having a subgroup of index $2$ is a semi-Cayley graph. Thus, our results on semi-Cayley graphs can also help us to find the existence of fractional revival on some Cayley graphs over a group having a subgroup of index $2$.

The paper is organized as follows. In Section 2, we introduce some known results on the eigenvalues and the spectral decomposition of the adjacency matrix of a semi-Cayley graph. In Section 3, we give some necessary and sufficient conditions for semi-Cayley graphs admitting fractional revival. In Section 4, we prove that integrality is necessary for semi-Cayley graphs admitting fractional revival. In Section 5, we characterize the minimum time when semi-Cayley graphs admit fractional revival. In Section 6, we give examples of certain Cayley graphs over the generalized dihedral groups and generalized dicyclic groups admitting fractional revival.

\section{Preliminaries}
In this section, we give some notions, notations and helpful results used in this paper.

\subsection{Characters of finite abelian groups}

Let $\mathbb{C}$ and $\mathbb{Z}$ denote the set of complex numbers and integer numbers, respectively.
Let $G$ be a finite abelian group of order $n$. It is well-known that $G$ can be decomposed as a direct sum of cyclic groups:
 $$G=\mathbb{Z}_{n_1}\otimes \mathbb{Z}_{n_2}\otimes\cdots\otimes \mathbb{Z}_{n_r}~(n_s\geq2),$$
where $\mathbb{Z}_m=(\mathbb{Z}/m\mathbb{Z},+)$ is a cyclic group of order $m$.
 For every $z=(z_1,\ldots, z_r)\in G$, $(z_s\in \mathbb{Z}_{n_s})$, the mapping
 \begin{equation}\label{character}
 \chi_z: G\rightarrow \mathbb{C}, ~\chi_z(g)=\prod_{s=1}^r\omega_{n_s}^{z_sg_s}, ~(\text{for~} g=(g_1,\ldots,g_r)\in G)
 \end{equation}
is a \emph{character} of $G$, where $\omega_{n_s}=\exp(2\pi \mathrm{i}/n_s)$ is a primitive $n_s$-th root of unity in $\mathbb{C}$. It is easy to verify that
$$
\chi_z(g)=\chi_g(z) \text{~~for~all~}z,g\in G.
$$
For a character $\chi_z$ and a subset $X$ of $G$, we denote
$$
\chi_z(X):=\sum\limits_{x\in X}\chi_z(x).
$$

All characters of $G$ form a group $\hat{G}=\{\chi_z|z\in G\}$ that is isomorphic to $G$ under the operation
$\chi_z\chi_g(x):=\chi_z(x)\chi_g(x)$, where $g,h,x\in G$. The identity of the group $\hat{G}$ is the principal character $\chi_1$ that maps every element of $G$ to $1$.
Moreover, all characters $\chi\in \hat{G}$ satisfy the ``First orthogonality relations'' in \cite{Steinberg12}, that is,
\begin{equation}\label{Sum=0-0}
\frac{1}{n}\sum\limits_{g\in G}\chi_i(g)\overline{\chi_j(g)}=\left\{
\begin{array}{cc}
1,    &i=j; \\[0.2cm]
0,    &i\neq j.
\end{array}\right.
\end{equation}
where $\overline{\chi_j(g)}$ denotes the complex conjugate of $\chi_j(g)$. Thus,
\begin{equation}\label{Sum=0-1}
 \sum\limits_{g\in G, ~i\not=1}\chi_i(g)=\sum\limits_{g\in G, ~j\not=1}\overline{\chi_j(g)}=0.
\end{equation}


Let $f:G\rightarrow \mathbb{C}$ be a complex-valued function on $G$. The Fourier transform $\hat{f}:\hat{G}\rightarrow \mathbb{C}$ is defined by
$$\hat{f}(\chi)=\sum\limits_{g\in G} f(g)\overline{\chi(g)},$$
then the Fourier inversion \cite{Steinberg12} is
$$f=\frac{1}{n}\sum\limits_{\chi\in \hat{G}}\hat{f}(\chi)\chi.$$

\subsection{Spectral decomposition of semi-Cayley graphs}

Let $\Gamma$ be a graph with adjacency matrix $A$.  The eigenvalues of $A$ are called the \emph{eigenvalues} of $\Gamma$. Suppose that $ \lambda_1\geq\lambda_2\geq\cdots\geq\lambda_n$ ($\lambda_i$ and $\lambda_j$ may be equal) are all eigenvalues of $\Gamma$  and $\xi_j$ is the eigenvector associated with $\lambda_{j}$, $j=1,2,\ldots,n$.  Let $\mathbf{x}^H$ denote the conjugate transpose of a column vector $\mathbf{x}$. Then, for each eigenvalue $\lambda_j$ of $\Gamma$, define
$$
E_{\lambda_j} =\mathbf{\xi}_j (\mathbf{\xi}_j)^H,
$$
which is usually called the \emph{eigenprojector} corresponding to  $\lambda_j$ of $G$. Note that $\sum_{j=1}^nE_{\lambda_j}=I_n$ (the identity matrix of order $n$). Then
\begin{equation}\label{spect1}
A=A\sum_{j=1}^nE_{\lambda_j} =\sum_{j=1}^nA\mathbf{\xi}_j (\mathbf{\xi}_j)^H  =\sum_{j=1}^n\lambda_j\mathbf{\xi}_j (\mathbf{\xi}_j)^H  =\sum_{j=1}^{n}\lambda_jE_{\lambda_j},
\end{equation}
which is called the \emph{spectral decomposition of $A$ with respect to the eigenvalues}  (see ``Spectral Theorem for Diagonalizable Matrices'' in \cite[Page 517]{MAALA}). Note that $E_{\lambda_j}^{2}=E_{\lambda_j}$ and $E_{\lambda_j}E_{\lambda_h}=\mathbf{0}$ for $j\neq h$, where $\mathbf{0}$ denotes the zero matrix. So, by (\ref{spect1}), we have
\begin{equation*}\label{SpecDec2-1}
H(t)=\sum_{k\geq 0}\dfrac{\imath^{k}A^{k}t^{k}}{k!}=\sum_{k\geq 0}\dfrac{\imath^{k}\left(\sum\limits_{j=1}^{n}\lambda_{j}^{k}E_{\lambda_j}\right)t^{k}}{k!} =\sum_{j=1}^{n}\exp(\imath t\lambda_{j})E_{\lambda_j}.
\end{equation*}

In \cite{Are22}, Arezoomand gave the spectral decomposition of the adjacency matrix of a semi-Cayley graph over an abelian group. Before to present this result, we follow some symbols used in \cite{Are22}, which are stated as follows.

Let $a_z^+, a_z^-, b_z^+, b_z^-~(z\in G)$ be complex numbers such that
\begin{equation}\label{abpm}
\begin{split}
&b_z^\pm\chi_z(S)=\frac{-x_z\pm\sqrt{x_z^2+4|\chi_z(S)|^2}}{2}a_z^\pm,\\[0.2cm]
&a_z^\pm\overline{\chi_z(S)}=\frac{x_z\pm\sqrt{x_z^2+4|\chi_z(S)|^2}}{2}b_z^\pm,
\end{split}
\end{equation}
where $\overline{\chi_z(S)}$ denotes the complex conjugate of $\chi_z(S)$ and $$x_z=\chi_z(R)-\chi_z(L).$$
It is known \cite{JamLi01} that $\chi_z(R)$ and $\chi_z(L)$ are real numbers, since $R$ and $L$ are inverse-closed. Thus $x_z\in \mathbb{R}$. If $\chi_z(S)=0$, then assume that $(a_z^+, b_z^+)=(1,0)$ and $(a_z^-,b_z^-)=(0,1)$.

\begin{lemma}\emph{(see \cite[Section 3]{Are22})}\label{SC-Eigen-111}
Let $\Gamma=\SC(G,R,L,S)$ be a semi-Cayley graph over an abelian group $G=\{g_1,g_2,\ldots, g_n\}$. Let $a_z^\pm, b_z^\pm$ be defined by (\ref{abpm}) and $\hat{G}=\{\chi_z|z\in G\}$. Then
\begin{itemize}
\item[\rm(a)]the eigenvalues of $\Gamma$ are
\begin{equation*}
\lambda_z^\pm=\frac{1}{2}\left(\chi_z(R)+\chi_z(L)\pm\sqrt{(\chi_z(R)-\chi_z(L))^2+4|\chi_z(S)|^2}\right),~z\in G;
\end{equation*}
\item[\rm(b)]the eigenvectors of $\Gamma$ associated with the eigenvalue $\lambda_z^\pm$ are
\begin{equation*}
\xi_z^\pm=\frac{1}{\sqrt{n(|a_z^\pm|^2+|b_z^\pm|^2)}}\left(a_z^\pm\chi_z(g_1^{-1}), \ldots a_z^\pm\chi_z(g_n^{-1}),b_z^\pm\chi_z(g_1^{-1}),\ldots, b_z^\pm\chi_z(g_n^{-1})  \right)^\top,
\end{equation*}
where $\mathbf{x}^\top$ denotes the transpose of a column vector $\mathbf{x};$
\item[\rm(c)] the eigenprojector of $\Gamma$ associated with the eigenvalue $\lambda_z^\pm$ are
\begin{equation*}
E_{\lambda_z}^\pm:=\xi_z^\pm\cdot(\xi_z^\pm)^H=\frac{1}{n(|a_z^\pm|^2+|b_z^\pm|^2)}\left(
\begin{array}{cc}
|a_z^\pm|^2F                      &a_z^\pm\overline{b_z^\pm}F\\[0.2cm]
b_z^\pm\overline{a_z^\pm}F        &|b_z^\pm|^2F
\end{array}
\right),~z\in G,
\end{equation*}
where $F=(\chi_z(g_r^{-1}g_s))_{1\leq r,s\leq n}$.
\end{itemize}
Moreover, the spectral decomposition of the adjacency matrix $A$ of $\Gamma$ is
$$A=\sum\limits_{z\in G}\sum\limits_\pm\lambda_z^\pm E_{\lambda_z}^\pm,
\text{~and~}H(t)=\sum\limits_{z\in G}\sum\limits_\pm\exp(\imath \lambda_z^\pm t)E_{\lambda_z}^\pm.$$
\end{lemma}

Let $c_z^+,c_z^-,d_z^+,d_z^-,e_z^+,e_z^-~(z\in G)$ be complex numbers such that
\begin{align*}
c_z^\pm=\frac{|a_z^\pm|^2}{|a_z^\pm|^2+|b_z^\pm|^2},~~d_z^\pm=\frac{|b_z^\pm|^2}{|a_z^\pm|^2+|b_z^\pm|^2},~~
e_z^\pm=\frac{a_z^\pm\overline{b_z^\pm}}{|a_z^\pm|^2+|b_z^\pm|^2}.
\end{align*}
If $\chi_z(S)\neq 0$, then
\begin{equation}\label{equation21}
\begin{split}
&c_z^\pm=\frac{\left|x_z\pm\sqrt{x_z^2+4|\chi_z(S)|^2}\right|^2}{\left|x_z\pm\sqrt{x_z^2+4|\chi_z(S)|^2}\right|^2+4|\chi_z(S)|^2},\\
&d_z^\pm=\frac{4|\chi_z(S)|^2}{\left|x_z\pm\sqrt{x_z^2+4|\chi_z(S)|^2}\right|^2+4|\chi_z(S)|^2},\\
&e_z^\pm=\pm2\chi_z(S)\frac{x_z+\sqrt{x_z^2+4|\chi_z(S)|^2}}{\left|x_z+\sqrt{x_z^2+4|\chi_z(S)|^2}\right|^2+4|\chi_z(S)|^2}.
\end{split}
\end{equation}

\noindent Suppose that $u=(g,r)$ and $v=(h,s)$ are two vertices of $\Gamma$, where $r,s\in \{0,1\}$. Let $H(t)_{u,v}$ denote the $(u,v)$-entry of the transition matrix $H(t)$. Then
\begin{equation}\label{htuv}
H(t)_{u,v}=\left\{
\begin{array}{ll}
\frac{1}{n}\sum\limits_{z\in G}\left(c_z^+\exp(\imath t \lambda_z^+ )+c_z^-\exp(\imath t\lambda_z^- )\right)\chi_z(g^{-1}h), &r=s=0,\\[0.4cm]
\frac{1}{n}\sum\limits_{z\in G}\left(d_z^+\exp(\imath t\lambda_z^+ )+d_z^-\exp(\imath t\lambda_z^-)\right)\chi_z(g^{-1}h), &r=s=1,\\[0.4cm]
\frac{1}{n}\sum\limits_{z\in G}\left(e_z^+\exp(\imath t\lambda_z^+)+e_z^-\exp(\imath t\lambda_z^- )\right)\chi_z(g^{-1}h), &r=0,s=1,\\[0.4cm]
\frac{1}{n}\sum\limits_{z\in G}\left(\overline{e_z^+}\exp(\imath t\lambda_z^+ )+\overline{e_z^-}\exp(\imath t\lambda_z^-)\right)\chi_z(g^{-1}h),&r=1,s=0.\\
\end{array}
\right.
\end{equation}

The following lemma gives some useful properties on $c_z^\pm$, $d_z^\pm$ and $e_z^\pm$.

\begin{lemma}\label{cz+cz-}\emph{(see \cite[Lemma 4.3]{Are22})}
Keep the above notations.
\begin{itemize}
\item[\rm (a)] If $\chi_z(S)=0$, then $c_z^+=d_z^-=1$, $c_z^-=d_z^+=e_z^+=e_z^-=0$.
\item[\rm (b)] If $\chi_z(S)\neq 0$, then $0<c_z^+,c_z^-,d_z^+,d_z^-<1$, and $e_z^++e_z^-=0$.
\item[\rm (c)] For each $z\in G$, we have
$c_z^++c_z^-=1, ~d_z^++d_z^-=1, ~|e_z^+|=|e_z^-|\leq\frac{1}{2},$ and
$$c_z^+ c_z^-=d_z^+ d_z^-=\frac{|\chi_z(S)|^2}{4|\chi_z(S)|^2+x_z^2}.$$
\end{itemize}
\end{lemma}

\section{Characterization of semi-Cayley graphs admitting fractional revival}


\subsection{An equivalent condition on semi-Cayley graphs admitting fractional revival}
Before proceeding, we introduce a result obtained by Chan et al. in \cite{CCTVZ19}.
\begin{lemma}\label{lemma4}\emph{(see \cite[Proposition 4.1]{CCTVZ19})}
If $(\alpha,\beta)$-revival occurs from $u$ to $v$ in a graph, then $(-\frac{\bar{\alpha}\beta}{\bar{\beta}},\beta)$-revival occurs from $v$ to $u$ at the same time.
\end{lemma}

Gao and Luo has given the adjacency matrix of semi-Cayley graphs in \cite{GaoLuo10}.

\begin{lemma}\emph{(see \cite[Lemma 3.1]{GaoLuo10})}\label{lemma5}
Let $\Gamma=\SC(G,R,L,S)$ be a semi-Cayley graph over an abelian group $G$. Let $A,B,C$ and $D$ be the adjacency matrices of $\Gamma$, $\Cay(G, S)$, $\Cay(G, R)$ and $\Cay(G, L)$, respectively. Then
\begin{align*}
A=\left(
\begin{array}{cc}
C    &B\\ [0.2mm]
B^T  &D
\end{array}
\right),
\end{align*}
where $B^T$ means the transpose of $B$.
\end{lemma}

Note that $H(t)$ is a unitary matrix. If $(\alpha,\beta)$-revival occurs on $\Gamma=\SC(G,R,L,S)$ from $u$ to $v$, by (\ref{FRCAY-EQUATION3}), the $u$-th row of $H(t)$ is determined by
\begin{equation*}
H(t)_{u,w}=\left\{
\begin{array}{ccc}
\alpha, &\text{if}~~w=u,\\[0.2cm]
\beta,  &\text{if}~~w=v,\\[0.2cm]
0,      &\text{otherwise}.
\end{array}\right.
\end{equation*}

Next, we obtain the transition matrix with the block form of the adjacency matrix described in Lemma \ref{lemma5}.

\begin{lemma}\label{lemma1}
Let $\Gamma=\SC(G,R,L,S)$ be a semi-Cayley graph over an abelian group $G$ of order $n$. Suppose that $\alpha$ and $\beta$ are two complex numbers. Then
\begin{itemize}
\item[\rm (a)] $(\alpha,\beta)$-revival occurs on $\Gamma$ from a vertex $u=(g,0)$ to a vertex $v=(h,0)$ at some time $t$ if and only if the transition matrix is of the form
\begin{equation}\label{equation32}
H(t)=\left(
\begin{array}{cc}
\alpha I_n+\beta Q_n &\mathbf{0}\\[0.2cm]
\mathbf{0}         &*          \\[0.2cm]
\end{array}\right),
\end{equation}
where $*$ is a unitary matrix of order $n$ and $Q_n$ is a permutation matrix with no fixed points and $Q_n^T=Q_n$. Moreover, $\alpha\bar{\beta}+\bar{\alpha}\beta=0$.
\item[\rm (b)]$(\alpha,\beta)$-revival occurs on $\Gamma$ from a vertex $u=(g,1)$ to a vertex $v=(h,1)$ at some time $t$ if and only if the transition matrix is of the form
\begin{equation*}\label{equation33}
H(t)=\left(
\begin{array}{cc}
*'                  &\mathbf{0}\\[0.2cm]
\mathbf{0}         &\alpha I_n+\beta Q_n'          \\[0.2cm]
\end{array}\right),
\end{equation*}
where $*'$ is a unitary matrix of order $n$ and $Q'_n$ is a permutation matrix with no fixed points and $Q_n'^T=Q_n'$. Moreover, $\alpha\bar{\beta}+\bar{\alpha}\beta=0$.
\item[\rm (c)]
$(\alpha,\beta)$-revival occurs on $\Gamma$ from a vertex $u=(g,0)$ to a vertex $v=(h,1)$ at some time $t$ if and only if the transition matrix is of the form
\begin{equation}\label{equation34}
H(t)=\left(
\begin{array}{cc}
\alpha I_n         &\beta Q_n''\\[0.2cm]
\beta Q_n''^T          &-\frac{\bar{\alpha}\beta}{\bar{\beta}} I_n          \\[0.2cm]
\end{array}\right),
\end{equation}
where $Q_n''$ is a permutation matrix with no fixed points.
\end{itemize}
\end{lemma}

\begin{proof}
Let $u=(g,r)$ and $v=(h,s)$ be two arbitrarily distinct vertices of $\Gamma$. Note that
$$a_{(g,r),(h,s)}=a_{(gz,r),(hz,s)}~\text{for}~g,h,z\in G ~\text{and}~r,s\in \{0,1\}.$$
Suppose that there exists a walk of length $k$ from $(g,r)$ to $(h,s)$, say $(g,r)\rightarrow(g_1,*)\rightarrow(g_2,*)\rightarrow\cdots\rightarrow(h,s)$. There must exist a walk of length $k$ from $(gz,r)$ to $(hz,s)$: $(gz,r)\rightarrow(g_1z,*)\rightarrow(g_2z,*)\rightarrow\cdots\rightarrow(hz,s)$. Conversely, suppose that there exists a walk of length $k$ from $(gz,r)$ to $(hz,s)$, say $(gz,r)\rightarrow(g_1',*)\rightarrow(g_2',*)\rightarrow\cdots\rightarrow(hz,s)$. There must exist a walk of length $k$ from $(g,r)$ to $(h,s)$: $(g,r)\rightarrow(g_1'z^{-1},*) \rightarrow(g_2'z^{-1},*)\rightarrow\cdots\rightarrow(h,s)$. Thus, we arrived that the number of walks of length $k$ from $(g,r)$ to $(h,s)$, is equal to the number of walks of length $k$ from $(gz,r)$ to $(hz,s)$. Let $a^{(k)}_{(g,r),(h,s)}$ denote the the number of walks of length $k$ from $(g,r)$ to $(h,s)$. Then
$$a^{(k)}_{(g,r),(h,s)}=a^{(k)}_{(gz,r),(hz,s)}~\text{for}~g,h,z\in G ~\text{and}~r,s\in \{0,1\},$$
Note \cite[Proposition~1.3.4]{Cvetkovic10} that $A^k=(a_{u,v}^{(k)})_{u,v\in V(\Gamma)}$, where $k\in \mathbb{Z}^+$, the set of positive integers.
With the definition of $H(t)$, we have
$$H(t)_{(g,r),(h,s)}=H(t)_{(gz,r),(hz,s)}.$$

(a) We first prove the necessity. Suppose that $(\alpha,\beta)$-revival occurs on $\Gamma$ from a vertex $u=(g,0)$ to a vertex $v=(h,0)$ at some time $t$, then
$$H(t)_{(g,0),(g,0)}=\alpha,~H(t)_{(g,0),(h,0)}=\beta, \text{~and~}H(t)_{(g,0),w}=0, ~\forall w\neq (g,0),(h,0),$$
which implies that
$$H(t)_{(gz,0),(gz,0)}=\alpha,~H(t)_{(gz,0),(hz,0)}=\beta, \text{~and~}H(t)_{(gz,0),w'}=0, ~\forall w'\neq (gz,0),(hz,0).$$
When $z$ runs through $G$, it is easy to get that (\ref{equation32})  holds, where $*$ is a unitary matrix of order $n$, and $Q_n$ is a permutation matrix without fixed points satisfying $Q_nQ_n^T=I_n$.  
Note that the matrix $Q_n$ is symmetric as $H(t)$ is a symmetric matrix. Then $Q_n^2=Q_nQ_n^T=I_n$. Notice also that $H(t)$ is a unitary matrix. Then
\begin{align*}
H(t)(H(t))^H&=\left(
\begin{array}{cc}
\alpha I_n+\beta Q_n   &0\\
0                      &*
\end{array}
\right)\left(
\begin{array}{cc}
\bar{\alpha} I_n+\bar{\beta} Q_n    &0\\
0                                   &*^H
\end{array}
\right)\\
&=\left(
\begin{array}{cc}
(\alpha I_n+\beta Q_n)(\bar{\alpha} I_n+\bar{\beta}Q_n)   &0\\
0                                                         &I_n
\end{array}
\right)\\
&=I_{2n}.
\end{align*} 
Thus,
$$(\alpha I_n+\beta Q_n)(\bar{\alpha} I_n+\bar{\beta}Q_n)=\alpha\bar{\alpha}I_n +(\alpha\bar{\beta}+\bar{\alpha}\beta)Q_n+\beta\bar{\beta}Q_n^2=I_n,$$
which implies that
$$\alpha\bar{\beta}+\bar{\alpha}\beta=0.$$


Next we prove the sufficiency. Suppose that $H(t)$ is of the form (\ref{equation32}). Since $Q_n$ is a permutation matrix with no fixed points, there are two elements $g$ and $h$ such that
$$
(Q_n)_{g,h}=1, \text{~and~}(Q_n)_{g,f}=0, ~\forall f\neq h.
$$
Thus,
$$
H(t)_{(g,0),(g,0)}=\alpha,~H(t)_{(g,0),(h,0)}=\beta, \text{~and~}H(t)_{(g,0),(f,0)}=0, ~\forall f\neq g,h.
$$
Hence $(\alpha,\beta)$-revival occurs from $(g,0)$ to $(h,0)$ at some time $t$.

(b) The proof is analogous to that of (a), and hence we omit the details here.

(c) We first prove the necessity. Suppose that $(\alpha,\beta)$-revival occurs on $\Gamma$ from a vertex $u=(g,0)$ to a vertex $v=(h,1)$ at some time $t$. Then
$$H(t)_{(g,0),(g,0)}=\alpha,~H(t)_{(g,0),(h,1)}=\beta, \text{~and~}H(t)_{(g,0),w}=0, ~\forall w\neq (g,0),(h,1),$$
which implies that
$$H(t)_{(gz,0),(gz,0)}=\alpha,~H(t)_{(gz,0),(hz,1)}=\beta, \text{~and~}H(t)_{(gz,0),w'}=0, ~\forall w'\neq (gz,0),(hz,1).$$
By Lemma \ref{lemma4}, we arrive that $(-\frac{\bar{\alpha}\beta}{\bar{\beta}},\beta)$-revival occurs on $\Gamma$ from  $v=(h,1)$ to  $u=(g,0)$ at the same time $t$. Thus,
$$H(t)_{(h,1),(h,1)}=-\frac{\bar{\alpha}\beta}{\bar{\beta}},~H(t)_{(h,1),(g,0)}=\beta, \text{~and~}H(t)_{(h,1),w}=0, ~\forall w\neq (g,0),(h,1),$$
which implies that
$$H(t)_{(hz,1),(hz,1)}=-\frac{\bar{\alpha}\beta}{\bar{\beta}},~H(t)_{(hz,1),(gz,0)}=\beta, \text{~and~}H(t)_{(hz,1),w'}=0, ~\forall w'\neq (gz,0),(hz,1).$$
When $z$ runs through $G$, it is easy to get that (\ref{equation34})  holds, where $Q_n''$ is a permutation matrix without fixed points satisfying $Q_n''Q_n''^T=I_n$.

Next we prove the sufficiency. Suppose that $H(t)$ is of the form (\ref{equation34}). Since $Q_n''$ is a permutation matrix with no fixed points, there are two elements $g$ and $h$ such that
$$
(Q_n'')_{g,h}=1, \text{~and~}(Q_n'')_{g,f}=0, ~\forall f\neq h.
$$
Thus,
$$
H(t)_{(g,0),(g,0)}=\alpha,~H(t)_{(g,0),(h,1)}=\beta, \text{~and~}H(t)_{(g,0),w}=0, ~\forall w\neq (g,0),(h,1).
$$
Hence $(\alpha,\beta)$-revival occurs from $(g,0)$ to $(h,1)$ at some time $t$.
\qed\end{proof}

If $(\alpha,\beta)$-revival occurs from a vertex $u=(g,r)$ to a vertex $v=(h,s)$ with $r=s$, and $u,v$ range over all vertices of $\Gamma$, then we say $\Gamma$ has $(\alpha,\beta)$-revival \emph{in the same orbit}. If $(\alpha,\beta)$-revival occurs from a vertex $u=(g,r)$ to a vertex $v=(h,s)$ with $r\neq s$, and $u,v$ range over all vertices of $\Gamma$, then we say $\Gamma$ has $(\alpha,\beta)$-revival \emph{between the different orbits}.

Note that (a) and (c) (respectively, (b) and (c)) in Lemma  \ref{lemma1} cannot occur at the same time. Then, we obtain the following result immediately.

\begin{cor}
Let $\Gamma=\SC(G,R,L,S)$ be a semi-Cayley graph over an abelian group $G$ of order $n$. Suppose that $\alpha$ and $\beta$ are two complex numbers. Then $\Gamma$ has $(\alpha,\beta)$-revival from every vertex if and only if either $\Gamma$ has $(\alpha,\beta)$-revival in the same orbit, or $\Gamma$ has $(\alpha,\beta)$-revival between the different orbits.
\end{cor}


%
\subsection{Semi-Cayley graphs admitting fractional revival from a vertex $u=(g,r)$ to a vertex $v=(h,s)$ with $r=s$}

In this section, we give a necessary and sufficient condition for $(\alpha,\beta)$-revival occurring on semi-Cayley graphs from a vertex $u=(g,r)$ to a vertex $v=(h,s)$ with $r=s$, where $\alpha\beta\neq0$.

\begin{theorem}\label{maintheorem}
Let $\Gamma=\SC(G,R,L,S)$ be a semi-Cayley graph over an abelian group $G$ of order $n$. Suppose that $\alpha$ and $\beta$ are two nonzero complex numbers. Let
$$
X=\{z\in G:\chi_z(S)=0\}.
$$
Then $(\alpha,\beta)$-revival occurs on $\Gamma$ from a vertex $u=(g,0)$ to a vertex $v=(h,0)$ at some time $t$ if and only if the following conditions hold:
\begin{itemize}
\item[\rm (a)]
$a=g^{-1}h$ is of order two, and then $n$ is even;
\item[\rm (b)]
for every $z\in G_0$, $\exp(\imath t\lambda_z^+)=\alpha+\beta$; for every $z\in G_1$, $\exp(\imath t\lambda_z^+)=\alpha-\beta$; for every $z\in G_0\setminus X$, $\exp(\imath t\lambda_z^-)=\alpha+\beta$; for every $z\in G_1\setminus X$, $\exp(\imath t\lambda_z^-)=\alpha-\beta$, where \begin{equation}\label{G0G1}
G_0=\{z\in G:\chi_a(z)=1\},~~~~G_1=\{z\in G:\chi_a(z)=-1\}.
\end{equation}
\end{itemize}
\end{theorem}

\begin{proof}
First, we prove the necessity. Suppose that $(\alpha,\beta)$-revival occurs on $\Gamma$ from a vertex $u=(g,0)$ to a vertex $v=(h,0)$.  Let $w=(f,0)$ be an arbitrarily vertex. By (\ref{htuv}), we have
\begin{align}\label{htuw}
H(t)_{u,w}=\frac{1}{n}\sum\limits_{z\in G}\left(c_z^+\exp(\imath t \lambda_z^+ )+c_z^-\exp(\imath t\lambda_z^- )\right)\chi_z(g^{-1}f)
=\left\{
\begin{array}{ccc}
\alpha, &\text{if}~~w=u,\\[0.2cm]
\beta,  &\text{if}~~w=v,\\[0.2cm]
0,      &\text{otherwise}.
\end{array}\right.
\end{align}
Denote $a=g^{-1}h$. According to (\ref{htuw}) and its Fourier inversion, we get
\begin{equation}\label{equation4}
c_z^+\exp(\imath t \lambda_z^+)+c_z^-\exp(\imath t\lambda_z^-)=\alpha+\beta\overline{\chi_a(z)},~~\forall z\in G.
\end{equation}
Hence,
\begin{align}\label{equation1}\nonumber
&\left(c_z^+\exp(\imath t \lambda_z^+)+c_z^-\exp(\imath t\lambda_z^-)\right)\left(\overline{c_z^+\exp(\imath t \lambda_z^+)+c_z^-\exp(\imath t\lambda_z^-)}\right)\\
&=(\alpha+\beta\overline{\chi_a(z)})(\overline{\alpha+\beta\overline{\chi_a(z)}}),~~\forall z\in G.
\end{align}
Since $c_z^+,c_z^-\in \mathbb{R}$, the left side of (\ref{equation1}) is
\begin{align*}
&\left(c_z^+\exp(\imath t \lambda_z^+)+c_z^-\exp(\imath t\lambda_z^-)\right)\left(c_z^+\exp(-\imath t \lambda_z^+)+c_z^-\exp(-\imath t\lambda_z^-)\right)\\
&={c_z^+}^2+c_z^+c_z^-\exp(\imath t (\lambda_z^+-\lambda_z^-)+c_z^-c_z^+\exp(-\imath t (\lambda_z^+-\lambda_z^-)+{c_z^-}^2\\
&={c_z^+}^2+{c_z^-}^2+2c_z^+c_z^-\cos(t (\lambda_z^+-\lambda_z^-))\\
&=(c_z^++c_z^-)^2-2c_z^+c_z^-+2c_z^+c_z^-\cos(t (\lambda_z^+-\lambda_z^-))\\
&=1-2c_z^+c_z^-(1-\cos(t (\lambda_z^+-\lambda_z^-))).
\end{align*}
The right side of (\ref{equation1}) is
$$(\alpha+\beta\overline{\chi_a(z)})(\bar{\alpha}+\bar{\beta}\chi_a(z))
=1+\overline{\chi_a(z)}\bar{\alpha}\beta+\chi_a(z)\alpha\bar{\beta}.$$
Therefore, we have
\begin{equation}\label{equation2}
-2c_z^+c_z^-(1-\cos(t (\lambda_z^+-\lambda_z^-)))=\overline{\chi_a(z)}\bar{\alpha}\beta+\chi_a(z)\alpha\bar{\beta},~~\forall z\in G.
\end{equation}
We claim that $\chi_a(z)=\overline{\chi_a(z)}$, which is proved by the following two cases.

\noindent\emph{Case 1.} If $z\in X$, then $\chi_z(S)=0$. By Lemma \ref{cz+cz-} (a), $c_z^+=1$, $c_z^-=0$. Thus
$$\overline{\chi_a(z)}\bar{\alpha}\beta+\chi_a(z)\alpha\bar{\beta}=0,~~z\in X.$$
Combing with $\alpha \bar{\beta}+\bar{\alpha}\beta=0$ in Lemma \ref{lemma1}, we have
$$\chi_a(z)=\overline{\chi_a(z)},~~z\in X.$$

\noindent\emph{Case 2.} If $z\notin X$, then $\chi_z(S)\not=0$. By (\ref{equation2}), we have
\begin{equation}\label{equation2-1-1}
\sum_{z\in G}-2c_z^+c_z^-(1-\cos(t (\lambda_z^+-\lambda_z^-))) = \sum_{z\in G}\left(\overline{\chi_a(z)}\bar{\alpha}\beta+\chi_a(z)\alpha\bar{\beta}\right).
\end{equation}
Note that $a\not=1$. By (\ref{Sum=0-1}), the right hand side of (\ref{equation2-1-1}) is $0$. By Lemma \ref{cz+cz-} (c), (\ref{equation2-1-1})  is mounted to
\begin{align}\label{sum1-cos}\nonumber
0
&=-2\sum\limits_{z\in G}\frac{|\chi_z(S)|^2}{4|\chi_z(S)|^2+x_z^2}\left(1-\cos(t(\lambda_z^+-\lambda_z^-))\right)\\
&=-2\sum\limits_{z\notin X}\frac{|\chi_z(S)|^2}{4|\chi_z(S)|^2+x_z^2}\left(1-\cos(t(\lambda_z^+-\lambda_z^-))\right).
\end{align}
Note that \begin{equation*}
\frac{|\chi_z(S)|^2}{4|\chi_z(S)|^2+x_i^2}\left(1-\cos(t(\lambda_z^+-\lambda_z^-))\right)\geq0,~~z\notin X.
\end{equation*}
Thus, (\ref{sum1-cos}) implies that
\begin{equation*}
\frac{|\chi_z(S)|^2}{4|\chi_z(S)|^2+x_i^2}\left(1-\cos(t(\lambda_z^+-\lambda_z^-))\right)=0,~~z\notin X.
\end{equation*}
Since
$$
\frac{|\chi_z(S)|^2}{4|\chi_z(S)|^2+x_i^2}\not=0,
$$
we have
\begin{equation}\label{1-cos=0}
1-\cos(t(\lambda_z^+-\lambda_z^-))=0, ~~z\notin X.
\end{equation}
By  (\ref{equation2}), we have
$$\overline{\chi_a(z)}\bar{\alpha}\beta+\chi_a(z)\alpha\bar{\beta}=0,~~z\notin X.$$
Combing with $\alpha \bar{\beta}+\bar{\alpha}\beta=0$ in Lemma \ref{lemma1} again, we get
$$\chi_a(z)=\overline{\chi_a(z)},~~z\notin X.$$

Now, $\chi_a(z)=\overline{\chi_a(z)}$ holds for every $z\in G$. Then
$$\chi_a(z)=\pm1,~~\forall z\in G.$$
Consequently, $a$ is of order two, and then $n$ is even. Thus $G$ can be written as $G=G_0\cup G_1$  and $|G_0|=|G_1|=\frac{n}{2}$, where $G_0$ and $G_1$ are defined in (\ref{G0G1}). By (\ref{equation4}), we have
\begin{align}\label{equation5}
c_z^+\exp(\imath t \lambda_z^+)+c_z^-\exp(\imath t \lambda_z^-)=\left\{
\begin{array}{cccc}
\alpha+\beta, &\text{if}~~z\in G_0,\\[0.2cm]
\alpha-\beta, &\text{if}~~z\in G_1.
\end{array}
\right.
\end{align}

\noindent For $z\notin X$, (\ref{1-cos=0}) implies that
\begin{equation}\label{equation10}
(\lambda_z^+-\lambda_z^-)t=2k\pi, ~~\text{for~some}~k\in \mathbb{Z},~\text{all}~z\notin X,
\end{equation}
that is,
$$
\exp(\imath t \lambda_z^+)=\exp(\imath t \lambda_z^-),~~ \forall z\notin X.
$$
Due to (\ref{equation5}) and Lemma \ref{cz+cz-} (c), it holds that
\begin{align*}
c_z^+\exp(\imath t \lambda_z^+)+c_z^-\exp(\imath t \lambda_z^-)=\exp(\imath t \lambda_z^+)=\exp(\imath t \lambda_z^-)=\left\{
\begin{array}{cccc}
\alpha+\beta, &\text{if}~~z\in G_0\setminus X,\\[0.2cm]
\alpha-\beta, &\text{if}~~z\in G_1\setminus X.
\end{array}
\right.
\end{align*}

\noindent For $z\in X$, recall Lemma \ref{cz+cz-} (a) that $c_z^+=1$, $c_z^-=0$. By (\ref{equation5}), we have
\begin{align*}
\exp(\imath t \lambda_z^+)=\left\{
\begin{array}{cccc}
\alpha+\beta, &\text{if}~~z\in G_0\cap X,\\[0.2cm]
\alpha-\beta, &\text{if}~~z\in G_1\cap X.
\end{array}
\right.
\end{align*}


Next, we prove the sufficiency. For an arbitrary vertex $u=(g,0)\in \Gamma$, recall Lemma \ref{cz+cz-} that $c_z^++c_z^-=1$ for $z\in G$ and $c_z^+=1,c_z^-=0$ for $z\in X$. By (\ref{htuv}), we have
\begin{align*}
H(t)_{u,u}&=\frac{1}{n}\sum\limits_{z\in G}\left(c_z^+\exp(\imath t \lambda_z^+ )+c_z^-\exp(\imath t\lambda_z^- )\right)\\
          &=\frac{1}{n}\left(\sum\limits_{z\in G_0\setminus X}(\alpha+\beta)+\sum\limits_{z\in G_1\setminus X}(\alpha-\beta)+\sum\limits_{z\in G_0\cap X}(\alpha+\beta)+\sum\limits_{z\in G_1\cap X}(\alpha-\beta)\right)\\
          &=\frac{1}{n}\left(\sum\limits_{z\in G_0}(\alpha+\beta)+\sum\limits_{z\in G_1}(\alpha-\beta)\right)\\
          &=\alpha.
\end{align*}

Let $v=(g\cdot a,0)$. Recall that $a$ is of order two. Then $u\not=v$. By (\ref{htuv}) and (\ref{G0G1}), we have
\begin{align*}
H(t)_{u,v}&=\frac{1}{n}\sum\limits_{z\in G}\left(c_z^+\exp(\imath t \lambda_z^+ )+c_z^-\exp(\imath t\lambda_z^- )\right)\chi_z(a)\\
          &=\frac{1}{n}\left(\sum\limits_{z\in G_0}(\alpha+\beta)\chi_z(a)+\sum\limits_{z\in G_1}(\alpha-\beta)\chi_z(a)\right)\\[0.2cm]
          &=\frac{1}{n}\left(\frac{n}{2}(\alpha+\beta)-\frac{n}{2}(\alpha-\beta)\right)\\[0.2cm]
          &=\beta.
\end{align*}

Let $w=(g\cdot b,0)$ for any $b\neq 1,a$. Then $w\not=u$ and $w\not=v$. By (\ref{htuv}), (\ref{G0G1}), (\ref{Sum=0-0}) and (\ref{Sum=0-1}), we have
\begin{align*}
H(t)_{u,w}&=\frac{1}{n}\sum\limits_{z\in G}\left(c_z^+\exp(\imath t \lambda_z^+ )+c_z^-\exp(\imath t\lambda_z^- )\right)\chi_z(b)\\
          &=\frac{1}{n}\left(\sum\limits_{z\in G_0}(\alpha+\beta)\chi_z(b)+\sum\limits_{z\in G_1}(\alpha-\beta)\chi_z(b)\right)\\[0.2cm]
          &=\frac{\alpha}{n}\sum\limits_{z\in G}\chi_z(b)+\frac{\beta}{n}\left(\sum\limits_{z\in G_0}\chi_z(b)-\sum\limits_{z\in G_1}\chi_z(b)\right)\\[0.2cm]
          &=\frac{\beta}{n}\left(\sum\limits_{z\in G_0\cup G_1}\frac{1+\overline{\chi_z(a)}}{2}\chi_z(b)-\sum\limits_{z\in G_0\cup G_1}\frac{1-\overline{\chi_z(a)}}{2}\chi_z(b)\right)\\[0.2cm]
          &=\frac{\beta}{2n}\left(\sum\limits_{z\in G}\left(\chi_z(b)+ \overline{\chi_z(a)}\chi_z(b)-\chi_z(b) +\overline{\chi_z(a)}\chi_z(b)\right)\right)\\[0.2cm]
          &=\frac{\beta}{n}\sum\limits_{z\in G}\overline{\chi_z(a)}\chi_z(b)\\[0.2cm]
          &=0.
\end{align*}

Let $w=(g\cdot b,1)$ for any $b\in G$.  Then $w\not= u$ and $w\not= v$. Note Lemma \ref{cz+cz-} that $e_z^+=e_z^-=0$ for $z\in X$, and  $e_z^++e_z^-=0$ for $z\in G$. By (\ref{htuv}), we have
\begin{align*}
H(t)_{u,w}&=\frac{1}{n}\sum\limits_{z\in G}\left(e_z^+\exp(\imath t \lambda_z^+ )+e_z^-\exp(\imath t\lambda_z^- )\right)\chi_z(b)\\
          &=\frac{1}{n}\left(\sum\limits_{z\in G_0\setminus X}(e_z^++e_z^-)(\alpha+\beta)\chi_z(b)+\sum\limits_{z\in G_1\setminus X}(e_z^++e_z^-)(\alpha-\beta)\chi_z(b)\right)\\[0.2cm]
          &=0.
\end{align*}

Therefore, $(\alpha,\beta)$-revival occurs on $\Gamma$ from a vertex $u=(g,0)$ to a vertex $v=(h,0)$ at time $t$.


This completes the proof.
\qed\end{proof}

\begin{theorem}\label{maintheorem-case2}
Let $\Gamma=\SC(G,R,L,S)$ be a semi-Cayley graph over an abelian group $G$ of order $n$. Suppose that $\alpha$ and $\beta$ are two nonzero complex numbers. Let
$$
X=\{z\in G:\chi_z(S)=0\}.
$$
Then $(\alpha,\beta)$-revival occurs on $\Gamma$ from a vertex $u=(g,1)$ to a vertex $v=(h,1)$ at some time $t$ if and only if the following conditions hold:
\begin{itemize}
\item[\rm (a)]
$a=g^{-1}h$ is of order two, and then $n$ is even;
\item[\rm (b)]
for every $z\in G_0$, $\exp(\imath t\lambda_z^-)=\alpha+\beta$; for every $z\in G_1$, $\exp(\imath t\lambda_z^-)=\alpha-\beta$; for every $z\in G_0\setminus X$, $\exp(\imath t\lambda_z^+)=\alpha+\beta$; for every $z\in G_1\setminus X$, $\exp(\imath t\lambda_z^+)=\alpha-\beta$, where \begin{equation*}
G_0=\{z\in G:\chi_a(z)=1\},~~~~G_1=\{z\in G:\chi_a(z)=-1\}.
\end{equation*}
\end{itemize}
\end{theorem}
\begin{proof}
The proof is similar to that of Theorem \ref{maintheorem}. Hence, we omit the details here.
\qed\end{proof}


If $R=L$, by Lemma \ref{SC-Eigen-111} (a), the eigenvalues of $\Gamma=\SC(G,R,R,S)$ are
$$\lambda_z^+=\chi_z(R)+|\chi_z(S)|,~~~ \lambda_z^-=\chi_z(R)-|\chi_z(S)|,$$
which implies that $\lambda_z^+=\lambda_z^-=\chi_z(R)$ if $z\in X$. In this case, the sufficient condition for Theorem \ref{maintheorem} is the same as the sufficient condition for Theorem \ref{maintheorem-case2}. Thus $(\alpha,\beta)$-revival occurs on $\Gamma$ from a vertex $u=(g,0)$ to a vertex $v=(h,0)$ if and only if $(\alpha,\beta)$-revival occurs on $\Gamma$ from a vertex $u=(g,1)$ to a vertex $v=(h,1)$. By Lemma \ref{lemma1} (a) and (b), we obtain the following result.

\begin{cor}\label{cor3}
Let $\Gamma=\SC(G,R,R,S)$ be a semi-Cayley graph over an abelian group $G$ of order $n$. Suppose that $\alpha$ and $\beta$ are two nonzero complex numbers. Then $(\alpha,\beta)$-revival occurs on $\Gamma$ from a vertex $u=(g,r)$ to a vertex $v=(h,s)$  with $r=s$ if and only if $\Gamma$ has $(\alpha,\beta)$-revival in the same orbit.
\end{cor}

\subsection{Semi-Cayley graphs admitting fractional revival from a vertex $u=(g,r)$ to a vertex $v=(h,s)$ with $r\neq s$}

\begin{theorem}\label{maintheorem2}
Let $\Gamma=\SC(G,R,L,S)$ be a semi-Cayley graph over an abelian group $G$ of order $n$. Suppose that $\alpha$ and $\beta$ are two nonzero complex numbers. Let
$$
X=\{z\in G:\chi_z(S)=0\}.
$$
Then $(\alpha,\beta)$-revival occurs on $\Gamma$ from a vertex $u=(g,0)$ to a vertex $v=(h,1)$ at some time $t$ if and only if for all $z\in G$, the following conditions hold:
\begin{itemize}
\item[\rm (a)] $X=\emptyset$ and $\chi_z(S)\chi_a(z)\in \mathbb{R}$ with $a=g^{-1}h$;
\item[\rm (b)]
$\exp(\imath t \lambda_z^+)=\alpha+c_z^-(e_z^+)^{-1}\beta \overline{\chi_a(z)}$ and $\exp(\imath t \lambda_z^-)=\alpha-c_z^+(e_z^+)^{-1}\beta \overline{\chi_a(z)}$, where $c_z^+$, $c_z^-, e_z^+$ are defined as (\ref{equation21}).
\end{itemize}
\end{theorem}

\begin{proof}
We first prove the necessity. Suppose that $(\alpha, \beta)$-revival occurs from a vertex $u=(g,0)$ to a vertex $v=(h,1)$. For $w=(f,0)$, by (\ref{htuv}), we have
\begin{align*}
H(t)_{u,w}=\frac{1}{n}\sum\limits_{z\in G}\left(c_z^+\exp(\imath t \lambda_z^+ )+c_z^-\exp(\imath t\lambda_z^- )\right)\chi_z(g^{-1}f)
=\left\{
\begin{array}{cc}
\alpha, &\text{if}~~w=u,\\[0.2cm]
0,      &\text{otherwise}.
\end{array}\right.
\end{align*}
For $w=(f,1)$,
\begin{align*}
H(t)_{u,w}=\frac{1}{n}\sum\limits_{z\in G}\left(e_z^+\exp(\imath t \lambda_z^+ )+e_z^-\exp(\imath t\lambda_z^- )\right)\chi_z(g^{-1}f)
=\left\{
\begin{array}{cc}
\beta, &\text{if}~~w=v,\\[0.2cm]
0,      &\text{otherwise}.
\end{array}\right.
\end{align*}
Denote $a=g^{-1}h$. According to the Fourier inversion, we get
\begin{align}
\label{equation22} c_z^+\exp(\imath t \lambda_z^+)+c_z^-\exp(\imath t \lambda_z^-)&=\alpha, ~z\in G,\\[0.2cm]
\label{equation22-11} e_z^+\exp(\imath t \lambda_z^+)+e_z^-\exp(\imath t \lambda_z^-)&=\beta\overline{\chi_a(z)}, ~z\in G.
\end{align}

We claim that $X=\emptyset$. By contradiction, suppose that $X\neq \emptyset$. Then there exists a $z\in X$ such that $\chi_z(S)=0$. By Lemma \ref{cz+cz-} (a), we have $e_z^+=e_z^-=0$. By (\ref{equation22-11}), $\beta \overline{\chi_a(z)}=0$, a contradiction to (\ref{character}) and $\beta\neq0$. Therefore, $X=\emptyset$, and then by Lemma (\ref{cz+cz-}) (b), $e_z^++e_z^-=0$. Then (\ref{equation22-11}) leads to
\begin{equation}\label{equation23}
\exp(\imath t \lambda_z^+)-\exp(\imath t \lambda_z^-)=(e_z^+)^{-1}\beta\overline{\chi_a(z)}.
\end{equation}
Recall Lemma (\ref{cz+cz-}) (c) that $c_z^++c_z^-=1$. By (\ref{equation22}) and (\ref{equation23}), we get
$$\exp(\imath t \lambda_z^+)=\alpha+c_z^-(e_z^+)^{-1}\beta \overline{\chi_a(z)} \text{~and~} \exp(\imath t \lambda_z^-)=\alpha-c_z^+(e_z^+)^{-1}\beta \overline{\chi_a(z)},$$
yielding (b).

Recall that $(\alpha, \beta)$-revival occurs from $u=(g,0)$ to $v=(h,1)$. By Lemma \ref{lemma1}, $(-\frac{\bar{\alpha}\beta}{\bar{\beta}}, \beta)$-revival occurs from $v=(h,1)$ to $u=(g,0)$. For $w=(f,0)$, by (\ref{htuv}), we have
\begin{align*}
H(t)_{v,w}=\frac{1}{n}\sum\limits_{z\in G}\left(\overline{e_z^+}\exp(\imath t \lambda_z^+ )+\overline{e_z^-}\exp(\imath t\lambda_z^- )\right)\chi_z(h^{-1}f)
=\left\{
\begin{array}{cc}
\beta, &\text{if}~~w=u,\\[0.2cm]
0,      &\text{otherwise}.
\end{array}\right.
\end{align*}
According to the Fourier inversion, we get
$$
\overline{e_z^+}\exp(\imath t \lambda_z^+ )+\overline{e_z^-}\exp(\imath t\lambda_z^- )=\beta\overline{\chi_z(h^{-1}g)}=\beta\chi_a(z),~z\in G.
$$
By Lemma (\ref{cz+cz-}) (b), $e_z^++e_z^-=0$. Then $\overline{e_z^-}=-\overline{e_z^+}$. Substituting this into the above equation,  we have
$$\exp(\imath t \lambda_z^+)-\exp(\imath t \lambda_z^-)=(\overline{e_z^+})^{-1}\beta\chi_a(z),~z\in G.$$
Combining with (\ref{equation23}), we have
$$(\overline{e_z^+})^{-1}\beta\chi_a(z)=(e_z^+)^{-1}\beta\overline{\chi_a(z)}.$$
Note that $\beta\neq 0$. Then
$$(\overline{e_z^+})^{-1}\chi_a(z)=(e_z^+)^{-1}\overline{\chi_a(z)}.$$
By (\ref{equation21}), we have
$$(\overline{\chi_z(S)})^{-1}\chi_a(z)=(\chi_z(S))^{-1}\overline{\chi_a(z)},$$
that is,
$$\chi_z(S)\chi_a(z)=\overline{\chi_z(S)} \overline{\chi_a(z)},$$
which implies that $\chi_z(S)\chi_a(z)\in \mathbb{R}$, yielding (a).


Next, we prove the sufficiency.
For a vertex $u=(g,0)\in \Gamma$, recall Lemma \ref{cz+cz-} that $c_z^++c_z^-=1$ for $z\in G$. By (\ref{htuv}), we have
\begin{align*}
H(t)_{u,u}&=\frac{1}{n}\sum\limits_{z\in G}\left(c_z^+\exp(\imath t \lambda_z^+ )+c_z^-\exp(\imath t\lambda_z^- )\right)\\
          &=\frac{1}{n}\sum\limits_{z\in G}\left(c_z^+\left(\alpha+c_z^-(e_z^+)^{-1}\beta\overline{\chi_a(z)}\right)
          +c_z^-\left(\alpha-c_z^+(e_z^+)^{-1}\beta\overline{\chi_a(z)}\right)\right)\\
          &=\frac{1}{n}\sum\limits_{z\in G}\left(c_z^+\alpha+c_z^-\alpha\right)\\
          &=\alpha.
\end{align*}

Let $v=(ga,1)$. Then $u\not= v$. Note that $X=\emptyset$, that is, $\chi_z(S)\not=0,~\forall z\in G$. Lemma (\ref{cz+cz-}) (b) implies that $e_z^++e_z^-=0$, that is, $e_z^-=- e_z^+$. By (\ref{htuv}), we have
\begin{align*}
H(t)_{u,v}&=\frac{1}{n}\sum\limits_{z\in G}\left(e_z^+\exp(\imath t\lambda_z^+)+e_z^-\exp(\imath t\lambda_z^- )\right)\chi_z(a)\\
&=\frac{1}{n}\sum\limits_{z\in G}e_z^+\left(\exp(\imath t \lambda_z^+ )-\exp(\imath t\lambda_z^- )\right)\chi_a(z)\\
          &=\frac{1}{n}\sum\limits_{z\in G}e_z^+\left(\alpha+c_z^-(e_z^+)^{-1}\beta\overline{\chi_a(z)}-
          \alpha+c_z^+(e_z^+)^{-1}\beta\overline{\chi_a(z)}\right)\chi_a(z)\\[0.2cm]
          &=\frac{1}{n}\sum\limits_{z\in G}e_z^+(e_z^+)^{-1}\beta\overline{\chi_a(z)}\chi_a(z)(c_z^++c_z^-)\\[0.2cm]
          &=\beta.
\end{align*}

Let $w=(gb,0)$, for every $b\neq 1$. Then $w\not= u$ and $w\not= v$. By (\ref{htuv}) and (\ref{Sum=0-1}), we have
\begin{align*}
H(t)_{u,w}&=\frac{1}{n}\sum\limits_{z\in G}\left(c_z^+\exp(\imath t \lambda_z^+ )+c_z^-\exp(\imath t\lambda_z^- )\right)\chi_z(b)\\
          &=\frac{1}{n}\sum\limits_{z\in G}\left(c_z^+(\alpha+c_z^-(e_z^+)^{-1}\beta\overline{\chi_a(z)})
          +c_z^-(\alpha-c_z^+(e_z^+)^{-1}\beta\overline{\chi_a(z)})\right)\chi_b(z)\\[0.2cm]
          &=\frac{\alpha}{n}\sum\limits_{z\in G}\chi_b(z)\\
          &=0.
\end{align*}

Let $w=(gb,1)$, for every $b\neq a$.  Then $w\not= u$ and $w\not= v$. By (\ref{htuv}) and (\ref{Sum=0-0}), we have
\begin{align*}
H(t)_{u,w}&=\frac{1}{n}\sum\limits_{z\in G}e_z^+\left(\exp(\imath t \lambda_z^+ )-\exp(\imath t\lambda_z^- )\right)\chi_z(b)\\
          &=\frac{1}{n}\sum\limits_{z\in G}e_z^+\left(\alpha+c_z^-(e_z^+)^{-1}\beta\overline{\chi_a(z)}-
          \alpha+c_z^+(e_z^+)^{-1}\beta\overline{\chi_a(z)}\right)\chi_b(z)\\[0.2cm]
          &=\frac{\beta}{n}\sum\limits_{z\in G}\overline{\chi_a(z)}\chi_b(z)\\[0.2cm]
          &=0.
\end{align*}

Therefore, $(\alpha,\beta)$-revival occurs on $\Gamma$ from a vertex $u=(g,0)$ to a vertex $v=(h,1)$ at time $t$.


This completes the proof.
\qed\end{proof}

\begin{theorem}\label{maintheorem2-case2}
Let $\Gamma=\SC(G,R,L,S)$ be a semi-Cayley graph over an abelian group $G$ of order $n$. Suppose that $\alpha$ and $\beta$ are two nonzero complex numbers. Let
$$
X=\{z\in G:\chi_z(S)=0\}.
$$
Then $(\alpha,\beta)$-revival occurs on $\Gamma$ from a vertex $u=(g,1)$ to a vertex $v=(h,0)$ at some time $t$ if and only if for all $z\in G$, the following conditions hold:
\begin{itemize}
\item[\rm (a)] $X=\emptyset$ and $\chi_z(S)\overline{\chi_a(z)}\in \mathbb{R}$ with $a=g^{-1}h$;
\item[\rm (b)]
$\exp(\imath t \lambda_z^+)=\alpha+d_z^-(\overline{e_z^+})^{-1}\beta \overline{\chi_a(z)}$ and $\exp(\imath t \lambda_z^-)=\alpha-d_z^+(\overline{e_z^+})^{-1}\beta \overline{\chi_a(z)}$, where $d_z^+$, $d_z^-, e_z^+$ are defined as (\ref{equation21}).
\end{itemize}
\end{theorem}
\begin{proof}
The proof is similar to that of Theorem \ref{maintheorem2}. Hence, we omit the details here.
\qed\end{proof}

If $\Gamma=\SC(G,R,L,S)$ is a semi-Cayley graph over an abelian group $G$ with $R=L$, Theorems \ref{maintheorem2} and \ref{maintheorem2-case2} imply the following result.

\begin{cor}\label{maincor}
Let $\Gamma=\SC(G,R,R,S)$ be a semi-Cayley graph over an abelian group $G$ of order $n$. Suppose that $\alpha$ and $\beta$ are two nonzero complex numbers. Let
$$
X=\{z\in G:\chi_z(S)=0\}
$$
 Then $(\alpha,\beta)$-revival occurs on $\Gamma$ from a vertex $u=(g,0)$ to a vertex $v=(h,1)$ at some time $t$(or $(\alpha,\beta)$-revival occurs on $\Gamma$ from vertex $v=(h,1)$ to vertex $u=(g,0)$ at some time $t$) if and only if for all $z\in G$, the following conditions hold:
\begin{itemize}
\item[\rm (a)] $X=\emptyset$ and $\chi_z(S)\chi_a(z)\in \mathbb{R}$ with $a=g^{-1}h$.
\item[\rm (b)] $\chi_a(z)=\pm\frac{|\chi_z(S)|}{\chi_z(S)}$  for every $z\in G$.
\item[\rm (c)] for every $z\in G\cap H_0$,
$\exp(\imath t \lambda_z^+)=\alpha+\beta$,~$\exp(\imath t \lambda_z^-)=\alpha-\beta$; for every $z\in G\cap H_1$, $\exp(\imath t \lambda_z^+)=\alpha-\beta$,~$\exp(\imath t \lambda_z^-)=\alpha+\beta$, where
\begin{equation}\label{equation24}
H_0=\{z\in\! G\!:\!\chi_z(S)\chi_a(z)\!=\!|\chi_z(S)|\},H_1=\{z\in\! G\!:\!\chi_z(S)\chi_a(z)\!=\!-|\chi_z(S)|\}.
\end{equation}
\end{itemize}
\end{cor}
\begin{proof}
If $u=(g,0)$ and $v=(h,1)$, then
(a) comes from Theorem \ref{maintheorem2} (a) immediately.

Recall (\ref{abpm}) that $x_z=\chi_z(R)-\chi_z(L)$. If $R=L$, then $x_z=0$. By (\ref{equation21}), we have $c_z^+=c_z^-=\frac{1}{2}$ and $e_z^+=\frac{\chi_z(S)}{2|\chi_z(S)|}$. Then
$$c_z^+(e_z^+)^{-1}\overline{\chi_a(z)}=c_z^-(e_z^+)^{-1}\overline{\chi_a(z)}=\frac{|\chi_z(S)|}{\chi_z(S)\chi_a(z)}\in \mathbb{R},~~\text{and~}\left|\frac{|\chi_z(S)|}{\chi_z(S)\chi_a(z)}\right|=1.$$
Therefore, $\chi_z(S)\chi_a(z)=\pm|\chi_z(S)|$, yielding (b).

By Theorem \ref{maintheorem2} (b), we have
\begin{align*}
\begin{array}{cc}
\exp(\imath t \lambda_z^+)=\alpha+\beta,~\exp(\imath t\lambda_z^-)=\alpha-\beta, &\text{~for~} z\in G\cap H_0,\\[0.2cm]
\exp(\imath t \lambda_z^+)=\alpha-\beta,~\exp(\imath t\lambda_z^-)=\alpha+\beta, &\text{~for~} z\in G\cap H_1,
\end{array}
\end{align*}
yielding (c).

If $u=(h,1)$ and $v=(g,0)$, the proof is similar to that of $u=(g,0)$ and $v=(h,1)$. Hence, we omit the details here.

This completes the proof.
\qed\end{proof}

The next result says that if a semi-Cayley graph $\Gamma=\SC(G,R,L,S)$ over an abelian group $G$ has $(\alpha,\beta)$-revival between the different orbits, then the condition $R=L$ is necessary.

\begin{prop}\label{mainprop2}
Let $\Gamma=\SC(G,R,L,S)$ be a semi-Cayley graph over an abelian group $G$ of order $n$. Suppose that $\alpha$ and $\beta$ are two nonzero complex numbers. If $\Gamma$ has $(\alpha,\beta)$-revival between the different orbits, then $R=L$.
\end{prop}

\begin{proof}
Recall that if $(\alpha,\beta)$-revival occurs from a vertex $u=(g,0)$ to a vertex $v=(h,1)$, then $(-\frac{\bar{\alpha}\beta}{\bar{\beta}}, \beta)$-revival occurs from vertex $v=(h,1)$ to vertex $u=(g,0)$. If $\Gamma$ has $(\alpha,\beta)$-revival between the different orbits, then $-\frac{\bar{\alpha}\beta}{\bar{\beta}}=\alpha$. Recall (\ref{equation22}) in Theorem \ref{maintheorem2} that
$$
c_z^+\exp(\imath t \lambda_z^+)+c_z^-\exp(\imath t \lambda_z^-)=\alpha,~z\in G.
$$
Using the similar method to obtain (\ref{equation22}), one can easily verify that
$$
d_z^+\exp(\imath t \lambda_z^+)+d_z^-\exp(\imath t \lambda_z^-)=-\frac{\bar{\alpha}\beta}{\bar{\beta}}=\alpha,~z\in G.
$$
Hence,
$$
(c_z^+-d_z^+)\exp(\imath t \lambda_z^+)=(d_z^--c_z^-)\exp(\imath t \lambda_z^-),~z\in G.
$$
Theorem \ref{maintheorem2} (a) implies that $\chi_z(S)\neq 0, ~\forall z\in G$. By (\ref{equation21}), we have
$$
c_z^+-d_z^+=d_z^--c_z^-=\frac{x_z}{\sqrt{4|\chi_z(S)|^2+x_z^2}}.
$$
Thus,
$$x_z(\exp(\imath t \lambda_z^+)-\exp(\imath t \lambda_z^-))=0,~z\in G.$$
If $\exp(\imath t \lambda_z^+)=\exp(\imath t \lambda_z^-)$, by Theorem \ref{maintheorem2} (b) and the fact that $c_z^++c_z^-=1$ in Lemma \ref{cz+cz-} (c), we have $(e_z^+)^{-1}\beta\overline{\chi_a(z)}=0$, a contradiction. Hence we have $x_z=0$ for all $z\in G$. That is, $\chi_z(R)=\chi_z(L)$ for all $z\in G$.

Define
$$
\delta_{i,j}=\left\{\begin{array}{ll}
                     1, &  \text{if~} i=j, \\[0.2cm]
                     0, & \text{if~} i\neq j.
                   \end{array}\right.
$$
By (\ref{Sum=0-0}), for any $x\in G$, we have
$$\sum\limits_{z\in G}\overline{\chi_z(x)}\chi_z(R)=\sum\limits_{r\in R}\sum\limits_{z\in G}\overline{\chi_z(x)}\chi_z(r)=n\sum\limits_{r\in R}\delta_{x,r},$$
and
$$\sum\limits_{z\in G}\overline{\chi_z(x)}\chi_z(L)=\sum\limits_{l\in L}\sum\limits_{z\in G}\overline{\chi_z(x)}\chi_z(l)=n\sum\limits_{l\in L}\delta_{x,l},$$
which implies that $\sum\limits_{r\in R}\delta_{x,r}=\sum\limits_{l\in L}\delta_{x,l}$ for all $x\in G$. Thus, $R=L$.
\qed\end{proof}

By Corollary \ref{maincor}, if $R=L$, then $(\alpha,\beta)$-revival occurs on $\Gamma$ from a vertex $u=(g,0)$ to a vertex $v=(h,1)$ if and only if $(\alpha,\beta)$-revival occurs on $\Gamma$ from $v=(h,1)$ to $u=(g,0)$.
By Lemma \ref{lemma1} (c), we have
$$\alpha=-\frac{\bar{\alpha}\beta}{\bar{\beta}}.$$
Thus, we conclude that $\Gamma$ has $(\alpha,\beta)$-revival between the different orbits.



\begin{cor}\label{cor2}
Let $\Gamma=\SC(G,R,L,S)$ be a semi-Cayley graph over an abelian group $G$ of order $n$. Suppose that $\alpha$ and $\beta$ are two nonzero complex numbers. Then $\Gamma$ has $(\alpha,\beta)$-revival between the different orbits if and only if $(\alpha,\beta)$-revival occurs on $\Gamma$ from a vertex $u=(g,r)$ to a vertex $v=(h,s)$ with $r\neq s$ and $R=L$.
\end{cor}

\section{The  integrality on semi-Cayley graphs admitting fractional revival}
A graph is said to be \emph{integral} if all the eigenvalues of its adjacency matrix are integers. In this section, we wonder whether a semi-Cayley graph admitting fractional revival is integral or not.


\begin{theorem}\label{cor1}
Let $\Gamma=\SC(G,R,R,S)$ be a semi-Cayley graph over an abelian group $G$ of order $n$. Suppose that $\alpha$ and $\beta$ are two nonzero complex numbers. If $(\alpha,\beta)$-revival occurs on $\Gamma$ from a vertex $u$ to a vertex $v$, then $\Gamma$ is an integral graph. Moreover, for every $z\in G$, $\chi_z(R)$ and $|\chi_z(S)|$ are integers.
\end{theorem}

\begin{proof}
If $(\alpha,\beta)$-revival occurs on $\Gamma$ from a vertex $u=(g,r)$ to vertex $v=(h,s)$ with $r=s$, by  Lemma \ref{lemma1} (a) and (b), we have $\bar{\alpha}\beta+\alpha\bar{\beta}=0$. If $(\alpha,\beta)$-revival occurs on $\Gamma$ from a vertex $u=(g,r)$ to a vertex $v=(h,s)$ with $r\neq s$, by Corollary \ref{cor2} and Lemma \ref{lemma1} (c), we also have $\bar{\alpha}\beta+\alpha\bar{\beta}=0$.

Assume that the complex numbers $\alpha=r_\alpha\exp(\imath\theta_\alpha)$ and $\beta=r_\beta\exp(\imath \theta_\beta)$. According to $|\alpha|^2+|\beta|^2=1$ and $\bar{\alpha}\beta+\alpha\bar{\beta}=0$, we have
\begin{align*}
\left\{
\begin{array}{lc}
r_\alpha^2+r_\beta^2=1, &\\[0.2cm]
\theta_\alpha-\theta_\beta=\frac{\pi}{2}+k\pi, &\text{for~some}~~k\in \mathbb{Z}.
\end{array}
\right.
\end{align*}
Thus
$$\alpha\pm\beta=\exp(\imath\theta_\beta)\left((-1)^k r_\alpha\imath\pm r_\beta\right).$$
Set $r_\beta+(-1)^k r_\alpha\imath=\exp(\imath \theta_0)$. Then $\theta_0\in \mathbb{R}$ and
\begin{equation}\label{equation9}
\alpha+\beta=\exp(\imath(\theta_0+\theta_\beta)), ~~\alpha-\beta=\exp(\imath(\pi-\theta_0+\theta_\beta)).
\end{equation}

\noindent \emph{Case 1.} Suppose that $(\alpha,\beta)$-revival occurs on $\Gamma$ from a vertex $u=(g,0)$ to a vertex $v=(h,0)$ at time $t$.
By Lemma \ref{SC-Eigen-111} (a), the eigenvalues of $\Gamma=\SC(G,R,R,S)$ are
$$\lambda_z^+=\chi_z(R)+|\chi_z(S)|,~~~ \lambda_z^-=\chi_z(R)-|\chi_z(S)|.$$
Set $t=2\pi T$. By (\ref{equation10}), we have
\begin{equation}\label{equation13}
2|\chi_z(S)|T\in \mathbb{Z},~~\forall z\notin X.
\end{equation}
Let $z=1\in G_0$, where $G_0$ is defined in (\ref{G0G1}). Then $\chi_1(S)=|S|\neq0$, which implies that $1\notin X$. By (\ref{equation13}), $2|\chi_1(S)|T=2|S|T\in \mathbb{Z}$. Thus $T\in \mathbb{Q}$, and then $|\chi_z(S)|\in \mathbb{Q}$ for $z\notin X$, where $\mathbb{Q}$ denotes the set of rational numbers. Note that $|\chi_z(S)|=0$ for $z\in X$. Therefore, $|\chi_z(S)|\in \mathbb{Q}$ for all $z\in G$. By Lemma \ref{SC-Eigen-111} (a), $|\chi_z(S)|$~($z\in G$) are eigenvalues of the semi-Cayley graph $\SC(G,\emptyset,\emptyset,S)$. Thus, $|\chi_z(S)|\in \mathbb{Z}$ for all $z\in G$.

On the other hand, by Theorem \ref{maintheorem} (b), we have
\begin{align}\label{equation8}
\exp(\imath t \lambda_z^+)=\left\{
\begin{array}{cccc}
\alpha+\beta, &\text{if}~~z\in G_0,\\[0.2cm]
\alpha-\beta, &\text{if}~~z\in G_1.
\end{array}
\right.
\end{align}
Set  $\theta_0=2\pi \mu_0$, $\theta_\beta=2\pi \mu_\beta$. Combining with (\ref{equation8}) and (\ref{equation9}), we have
\begin{align}\label{equation14}
\begin{array}{lc}
T\lambda_z^+-\mu_0-\mu_\beta\in \mathbb{Z}, &\forall z\in G_0, \\[0.2cm]
T\lambda_z^++\mu_0-\mu_\beta\in \frac{1}{2}+\mathbb{Z}, &\forall z\in G_1.
\end{array}
\end{align}
Therefore,
\begin{equation}\label{equation11}
T\sum\limits_{z\in G}\lambda_z^+-n\mu_\beta\in \frac{n}{4}+\mathbb{Z}.
\end{equation}
Note that
\begin{align*}
T\sum\limits_{z\in G}\lambda_z^+&=T\sum\limits_{z\in G}\left(\chi_z(R)+|\chi_z(S)|\right)\\
                               &=T\left(\sum\limits_{z\in G}\sum\limits_{r\in R}\chi_z(r)+\sum\limits_{z\in G}|\chi_z(S)|\right)\\
                               &=T\left(\sum\limits_{r\in R}\sum\limits_{z\in G}\chi_z(r)+\sum\limits_{z\in G}|\chi_z(S)|\right)\\
                               &=T\sum\limits_{z\in G}|\chi_z(S)|\in \mathbb{Q}.
\end{align*}
By (\ref{equation11}), we have $\mu_\beta\in \mathbb{Q}$.

Let $z=1\in G_0$, where $G_0$ is defined in (\ref{G0G1}). By (\ref{equation14}), we have
$$T(|R|+|S|)-\mu_0-\mu_\beta\in \mathbb{Z},$$
which implies that $\mu_0\in \mathbb{Q}$. By (\ref{equation14}) again, we have $\lambda_z^+\in \mathbb{Q}$ for all $z\in G$. Since $\lambda_z^+$~($z\in G$) are algebraic integers, we know that $\lambda_z^+\in \mathbb{Z}$ for all $z\in G$. Recall that $|\chi_z(S)|\in \mathbb{Z}$ for all $z\in G$, we have $\chi_z(R)\in \mathbb{Z}$ and then $\lambda_z^-\in \mathbb{Z}$ for all $z\in G$.

\noindent \emph{Case 2.} Suppose that $(\alpha,\beta)$-revival occurs on $\Gamma$ from a vertex $u=(g,1)$ to a vertex $v=(h,1)$ at time $t$. The proof is similar to that of Case 1, and hence we omit the details here.

\noindent \emph{Case 3.} Suppose that $(\alpha,\beta)$-revival occurs on $\Gamma$ from a vertex $u=(g,r)$ to a vertex $v=(h,s)$  at time $t$ with $r\neq s$. By Corollary \ref{maincor}, we have
\begin{align}\label{equation25}
\begin{array}{cc}
\exp(\imath t \lambda_z^+)=\alpha+\beta,~\exp(\imath t\lambda_z^-)=\alpha-\beta, &\text{~for~} z\in G\cap H_0,\\[0.2cm]
\exp(\imath t \lambda_z^+)=\alpha-\beta,~\exp(\imath t\lambda_z^-)=\alpha+\beta, &\text{~for~} z\in G\cap H_1.
\end{array}
\end{align}
Set $t=2\pi T$, $\theta_0=2\pi \mu_0$, $\theta_\beta=2\pi \mu_\beta$. Combining with (\ref{equation25}) and (\ref{equation9}), we have
\begin{align}\label{equation27}
\begin{array}{llll}
T\lambda_z^+-\mu_0-\mu_\beta\in \mathbb{Z},&T\lambda_z^-+\mu_0-\mu_\beta\in \frac{1}{2}+\mathbb{Z},&\text{~for~} z\in G\cap H_0, \\[0.2cm]
T\lambda_z^++\mu_0-\mu_\beta\in \frac{1}{2}+\mathbb{Z}, &T\lambda_z^--\mu_0-\mu_\beta\in \mathbb{Z},&\text{~for~} z\in G\cap H_1.
\end{array}
\end{align}
Thus,
$$T(\lambda_z^++\lambda_z^-)-2\mu_\beta\in \frac{1}{2}+\mathbb{Z}, ~z\in G.$$
Recall that $\lambda_z^+=\chi_z(R)+|\chi_z(S)|$ and $\lambda_z^-=\chi_z(R)-|\chi_z(S)|$. Then
\begin{equation}\label{equation26}
2T\chi_z(R)-2\mu_\beta\in \frac{1}{2}+\mathbb{Z}, ~z\in G.
\end{equation}
Note that
$$\sum\limits_{z\in G}(2T\chi_z(R)-2\mu_\beta)=2T\sum\limits_{z\in G}\chi_z(R)-2n\mu_\beta=-2n\mu_\beta\in \frac{n}{2}+\mathbb{Z}.$$
Thus, we have $\mu_\beta\in \mathbb{Q}$.

Let $z=1\in G\cap H_0$. By (\ref{equation26}), we have
$$2T|R|-2\mu_\beta\in \frac{1}{2}+\mathbb{Z},$$
which implies that $T\in \mathbb{Q}$.
By (\ref{equation27}), we have
$$T(|R|+|S|)-\mu_0-\mu_\beta\in \mathbb{Z}.$$
Thus $\mu_0\in \mathbb{Q}$. By (\ref{equation27}) again, we have $\lambda_z^+, \lambda_z^-\in \mathbb{Q}$ for all $z\in G$. Since $\lambda_z^+$, $\lambda_z^-$~($z\in G$) are algebraic integers, we know that $\lambda_z^+, \lambda_z^-\in \mathbb{Z}$ for all $z\in G$.
Since $\lambda_z^+-\lambda_z^-=2|\chi_z(S)|\in \mathbb{Z}$, we have $|\chi_z(S)|\in \mathbb{Q}$. By Lemma \ref{SC-Eigen-111} (a), $|\chi_z(S)|$~($z\in G$) are eigenvalues of the semi-Cayley graph $\SC(G,\emptyset,\emptyset,S)$. Thus, $|\chi_z(S)|\in \mathbb{Z}$ for all $z\in G$, and then $\chi_z(R)\in \mathbb{Z}$.

This completes the proof. \qed\end{proof}

\section{The minimum time on semi-Cayley graphs admitting fractional revival}

Let $\Gamma=\SC(G,R,L,S)$ be a semi-Cayley graph over an abelian group $G$ of order $n$, which admits fractional revival. If $R=L$, by Theorem \ref{cor1}, $\Gamma$ is an integral graph. By Lemma \ref{SC-Eigen-111} (a), the eigenvalues of $\Gamma=\SC(G,R,R,S)$ are
$$\lambda_z^+=\chi_z(R)+|\chi_z(S)|,~~~ \lambda_z^-=\chi_z(R)-|\chi_z(S)|.$$
In particular, if $z\in X$, where $X$ is defined in Theorem \ref{maintheorem}, we have
\begin{equation}\label{equation35}
\lambda_z^+=\lambda_z^-=\chi_z(R).
\end{equation}
Thus, in the context of $R=L$, Theorem \ref{maintheorem} (b) and Theorem \ref{maintheorem-case2} (b) are equivalent to
\begin{align}\label{equation15}
\exp(\imath t \lambda_z^\pm)=\left\{
\begin{array}{cccc}
\alpha+\beta, &\text{if}~~z\in G_0,\\[0.2cm]
\alpha-\beta, &\text{if}~~z\in G_1.
\end{array}
\right.
\end{align}
Set $t=2\pi T$, and take the identity element $z=1$ in $G_0$ and one fixed element $z_1$ in $G_1$, where $G_0, G_1$ are defined in (\ref{G0G1}). By (\ref{equation15}), we have
\begin{align}\label{equation16}
\begin{array}{llll}
T(|R|+|S|-\lambda_z^+)\in\mathbb{Z},         &T(|R|-|S|-\lambda_z^-)\in\mathbb{Z},                   &~~~\text{for~every~}z\in G_0,\\[0.2cm]
T(\lambda_{z_1}^+-\lambda_z^+)\in \mathbb{Z},&T(\lambda_{z_1}^--\lambda_z^-)\in \mathbb{Z}, &~~~\text{for~every~}z\in G_1.
\end{array}
\end{align}
Define the following integers:
\begin{align*}
&M_0^+=\gcd(|R|+|S|-\lambda_z^+:z\in G_0),~~~~
M_0^-=\gcd(|R|-|S|-\lambda_z^-:z\in G_0),\\[0.2cm]
&M_1^+=\gcd(\lambda_{z_1}^+-\lambda_z^+:z\in G_1),
~~~~~~~~~~~M_1^-=\gcd(\lambda_{z_1}^--\lambda_z^-:z\in G_1),
\end{align*}
and
\begin{equation}\label{equation17}
M=\gcd(M_0^+,M_0^-,M_1^+,M_1^-),
\end{equation}
where $\gcd(a_1,a_2,a_3,\ldots)$ denotes the greatest common divisor of $a_1,a_2,a_3,\ldots$.
Then (\ref{equation16}) is equivalent to
$$TM_0^+,~ TM_1^+,~ TM_0^-,~ TM_1^-\in \mathbb{Z},$$
and thus
\begin{equation}\label{equation30}
T\in \frac{1}{M}\mathbb{Z}.
\end{equation}

Next we prove that $M$ is a divisor of $2n=2|G|$.

\begin{lemma}\label{lemma2}
Let $G$ be an abelian group of order $n$. If  $(\alpha,\beta)$-revival occurs on $\Gamma=\SC(G,R,R,S)$ from a vertex $u=(g,r)$ to a vertex $v=(h,s)$  at time $t=2\pi T$ with $r=s$, then $M$ is a divisor of $2n$, where $M$ is defined by (\ref{equation17}).
\end{lemma}

\begin{proof}
Write
\begin{align*}
&|R|+|S|-\lambda_z^+=Mt_z,~~~~
|R|-|S|-\lambda_z^-=Ms_z, ~~~~   \text{for~every~}z\in G_0,\\[0.2cm]
&~\lambda_{z_1}^+-\lambda_z^+=Mr_z,
~~~~~~~~~~~\lambda_{z_1}^--\lambda_z^-=Mq_z,  ~~~~~~~~~~~\text{for~every~}z\in G_1,
\end{align*}
where $t_z,s_z,r_z,q_z\in \mathbb{Z}$.

Take an element $x\in R$ and $x\neq 1,a(=g^{-1}h)$. By (\ref{equation35}) and (\ref{Sum=0-0}), we have
\begin{align}\label{equation36}
\sum\limits_{z\in G} (\lambda_z^++\lambda_z^-)\overline{\chi_z(x)}=\sum\limits_{z\in G}2\chi_z(R)\overline{\chi_z(x)}
=2\sum\limits_{r\in R}\sum\limits_{z\in G}\chi_z(r)\overline{\chi_z(x)}
=2n.
\end{align}
On the other hand,
\begin{align}\label{equation37}\nonumber
&\sum\limits_{z\in G}(\lambda_z^++\lambda_z^-)\overline{\chi_z(x)}\\ \nonumber
&=\sum\limits_{z\in G_0}(\lambda_z^++\lambda_z^-)\overline{\chi_z(x)}
+\sum\limits_{z\in G_1}(\lambda_z^++\lambda_z^-)\overline{\chi_z(x)}\\ \nonumber
&=\sum\limits_{z\in G_0}(|R|+|S|-Mt_z+|R|-|S|-Ms_z)\overline{\chi_z(x)}
+\sum\limits_{z\in G_1}(\lambda_{z_1}^+-Mr_z+\lambda_{z_1}^--Mq_z)\overline{\chi_z(x)}\\
&=\sum\limits_{z\in G_0}(2|R|-M(t_z+s_z))\overline{\chi_z(x)}+\sum\limits_{z\in G_1}(\lambda_{z_1}^++\lambda_{z_1}^--M(r_z+q_z))\overline{\chi_z(x)}.
\end{align}
Moreover, by (\ref{Sum=0-0}) and (\ref{Sum=0-1}), we have
\begin{align*}
\sum\limits_{z\in G_0}2|R|\overline{\chi_z(x)}&=\sum\limits_{z\in G_0\cup G_1}\frac{1+\chi_z(a)}{2}\cdot2|R|\overline{\chi_z(x)}=0,\\
\sum\limits_{z\in G_1}(\lambda_{z_1}^++\lambda_{z_1}^-)\overline{\chi_z(x)}&=\sum\limits_{z\in G_0\cup G_1}\frac{1-\chi_z(a)}{2}(\lambda_{z_1}^++\lambda_{z_1}^-)\overline{\chi_z(x)}=0.
\end{align*}
Combining with (\ref{equation36}) and (\ref{equation37}), we have
\begin{equation}\label{equation18}
\frac{2n}{M}=-\sum\limits_{z\in G_0}(t_z+s_z)\overline{\chi_z(x)}-\sum\limits_{z\in G_1}(r_z+q_z)\overline{\chi_z(x)}.
\end{equation}
Since $t_z,s_z,r_z,q_z\in \mathbb{Z}$ and $\overline{\chi_z(x)}~(z\in G)$ are algebraic integers, the right hand of (\ref{equation18}) is an algebraic integer.
Then $2n/M\in \mathbb{Q}$ is an algebraic integer.  Thus $2n/M\in \mathbb{Z}$. This completes the proof.
\qed\end{proof}

Similar to the above discussion, set $t=2\pi T$, and take the identity element $z=1$ in $H_0$ and one fixed element $z_1$ in $H_1$, where $H_0, H_1$ are defined in (\ref{equation24}). By Corollary \ref{maincor}, we have
\begin{align}\label{equation28}
\begin{array}{llll}
T(|R|+|S|-\lambda_z^+)\in\mathbb{Z},         &T(|R|-|S|-\lambda_z^-)\in\mathbb{Z},                   &~~~\text{for~every~}z\in H_0,\\[0.2cm]
T(|R|+|S|-\lambda_z^-)\in\mathbb{Z},         &T(|R|-|S|-\lambda_z^+)\in\mathbb{Z}, \\[0.2cm]
T(\lambda_{z_1}^+-\lambda_z^+)\in \mathbb{Z},&T(\lambda_{z_1}^--\lambda_z^-)\in \mathbb{Z},
&~~~\text{for~every~}z\in H_1.
\end{array}
\end{align}
Define the following integers:
\begin{align*}
&N_0^+=\gcd(|R|+|S|-\lambda_z^+:z\in H_0),~~~~
N_0^-=\gcd(|R|-|S|-\lambda_z^-:z\in H_0),\\[0.2cm]
&N_1^+=\gcd(|R|+|S|-\lambda_z^-:z\in H_1),~~~~
N_1^-=\gcd(|R|-|S|-\lambda_z^+:z\in H_1),\\[0.2cm]
&N_2^+=\gcd(\lambda_{z_1}^+-\lambda_z^+:z\in H_1),
~~~~~~~~~~~N_2^-=\gcd(\lambda_{z_1}^--\lambda_z^-:z\in H_1),
\end{align*}
and
\begin{equation}\label{equation29}
N=\gcd(N_0^+,N_0^-,N_1^+,N_1^-,N_2^+,N_2^-).
\end{equation}
\noindent Then (\ref{equation28}) is equivalent to
$$TN_0^+,~ TN_0^-,~TN_1^+,~ TN_1^-,~TN_2^+,~ TN_2^-\in \mathbb{Z},$$
and thus
\begin{equation*}\label{equation31}
T\in \frac{1}{N}\mathbb{Z}.
\end{equation*}

\begin{lemma}\label{lemma3}
Let $G$ be an abelian group of order $n$. If $(\alpha,\beta)$-revival occurs on $\Gamma=\SC(G,R,R,S)$ from a vertex $u=(g,r)$ to a vertex $v=(h,s)$ at time $t=2\pi T$ with $r\neq s$, then $N$ is a divisor of $2n$, where $N$ is defined by (\ref{equation29}).
\end{lemma}

\begin{proof}
The proof is analogous to that of Lemma \ref{lemma2}. Hence we omit the proof here. \qed\end{proof}

By Lemmas \ref{lemma2} and \ref{lemma3}, one can easily verify the following result.

\begin{cor}
Let $G$ be an abelian group of order $n$. If $(\alpha,\beta)$-revival occurs on $\Gamma=\SC(G,R,R,S)$ from a vertex $u=(g,r)$ to a vertex $v=(h,s)$ at some time $t$, then $\exp(\imath t)$ is an $2n$-th root of unity.
\end{cor}



We end this section by giving two sufficient conditions for semi-Cayley graphs to admit no fractional revival.

\begin{theorem}
Let $M$ and $N$ be defined by (\ref{equation17}) and (\ref{equation29}), respectively.
\begin{itemize}
  \item[\rm (a)] If $M=1$, then $(\alpha, \beta)$-revival does not occur on $\Gamma=\SC(G,R,R,S)$ from a vertex $u=(g,r)$ to a vertex $v=(h,s)$ with $r=s$.
  \item[\rm (b)] If $N=1$, then $(\alpha, \beta)$-revival does not occur on $\Gamma=\SC(G,R,R,S)$ from a vertex $u=(g,r)$ to a vertex $v=(h,s)$ with $r\neq s$.
\end{itemize}
\end{theorem}

\begin{proof}
Assume that $(\alpha,\beta)$-revival occurs on $\Gamma=\SC(G,R,R,S)$ from a vertex $u=(g,r)$ to a vertex $v=(h,s)$ with $r=s$. By (\ref{equation30}), $T\in \mathbb{Z}$. By Lemma \ref{SC-Eigen-111} and Theorem \ref{cor1}, we have
$$H(t)=\sum\limits_{z\in G}\sum\limits_\pm\exp(\imath \lambda_z^\pm t)E_{\lambda_z}^\pm=\sum\limits_{z\in G}\sum\limits_\pm E_{\lambda_z}^\pm=I_n,$$
a contradiction to $\beta\neq 0$, yielding (a). Similarly, one can easily verify (b). This completes the proof.
\qed\end{proof}

\section{Examples}

%
%
%
%
%


In this section, we give some examples on semi-Cayley graphs having or not having $(\alpha, \beta)$-revival, which are actually Cayley graphs over generalized dihedral groups or generalized dicyclic groups.

Let $H$ be a finite abelian group. \emph{The generalized dihedral group} $\Dih(H)$ is defined as
$$\Dih(H)=\langle H,x|x^2=1, x^{-1}hx=h^{-1},  h\in H\rangle.$$
Let $y$ be an element of $H$ of order $2$. \emph{The generalized dicyclic group} $\Dic(H,y)$ is defined as
$$\Dic(H,y)=\langle H,x|x^2=y, x^{-1}hx=h^{-1},  h\in H\rangle.$$

The following result comes from \cite[Corollary 4.8]{Are22}.

\begin{lemma}\label{cay-semicay-11}
Let $\Gamma=\Cay(G,T)$ and $\Sigma=\SC(H, T_1, T_1, T_2)$, where $G=\Dih(H)$ or $G=\Dic(H,y)$, and $T=T_1\cup x T_2$ for some $T_1, T_2\subseteq H$. Set a bijection $\psi:V(\Gamma)\rightarrow V(\Sigma)$, by $\forall g\in G$, $\psi(g)=(g,0)$ and $\psi(xg)=(g,1)$. Then $\psi$ is an isomorphism from $\Gamma=\Cay(G,T)$ to $\Sigma=\SC(H, T_1, T_1, T_2)$.
\end{lemma}



\begin{example}
{\em
Let $G=\Dih(H)$ or $G=\Dic(H,y)$ and $\Gamma=\Cay(G,T)$ with $T=xH$.  By Lemma \ref{cay-semicay-11}, $\Gamma=\Cay(G,T)$ is isomorphic to $\SC(H, T_1, T_1, T_2)$ with $T_1=\emptyset$ and $T_2=H$. By (\ref{Sum=0-1}), we have
$$\chi_z(H)=\sum\limits_{h\in H}\chi_z(h)=0,  \forall  z\in H\setminus\{1\}.$$
Hence
$$X=\{z\in H:\chi_z(H)=0\}=H\setminus\{1\}.$$
Thus by Lemma \ref{SC-Eigen-111}, we have
$$\lambda_1^+=|H|, \lambda_1^-=-|H|, \lambda_z^+=\lambda_z^-=0, \forall  z\in H\setminus\{1\}.$$
Since there is no $\alpha\neq0, \beta\neq0$ such that
$$\alpha+\beta=\exp(\imath t \lambda_z^+)=1=\exp(\imath t \lambda_z^-)=\alpha-\beta, z\in H\setminus\{1\},$$
by Theorem \ref{maintheorem} (b) and Theorem \ref{maintheorem-case2} (b), $(\alpha, \beta)$-revival does not occur on $\Gamma$ from a vertex $g$ to a vertex $h$ for $1\neq g^{-1}h\in H$. Since $X\neq \emptyset$, by Corollary \ref{maincor}, $(\alpha, \beta)$-revival does not occur on $\Gamma$ from a vertex $g$ to a vertex $h$ for $g^{-1}h\notin H$.
}
\end{example}

\begin{example}
{\em
Let $G=\Dih(\mathbb{Z}_n)$ or $G=\Dic(\mathbb{Z}_n,y)$ and $\Gamma=\Cay(G,T)$ with $T=x\mathbb{Z}_n\cup\{n/2\}$. By Lemma \ref{cay-semicay-11}, $\Gamma=\Cay(G,T)$ is isomorphic to $\SC(H, T_1, T_1, T_2)$ with $T_1=\{n/2\}$ and $T_2=\mathbb{Z}_n$. Similarly, we have $X=\{z\in \mathbb{Z}_n:\chi_z(\mathbb{Z}_n)=0\}=\mathbb{Z}_n\setminus\{0\}$. Thus $$\lambda_0^+=|T_1|+|\mathbb{Z}_n|=1+n, \lambda_0^-=|T_1|-|\mathbb{Z}_n|=1-n, \lambda_z^+=\lambda_z^-=\chi_{z}(n/2),\forall z\in \mathbb{Z}_n\setminus\{0\}.$$
Let $g$ and $h$ be two vertices of $\Gamma$ with $a=n/2=g^{-1}h$. Then
$$G_0=\{z\in \mathbb{Z}_n:2|z\}, G_1=\{z\in \mathbb{Z}_n:2\nmid z\}.$$
Thus
$M_0^+=n$, $M_0^-=-n$, $M_1^+=M_1^-=0$. Then $M=n$. There is a minimum time $t=\frac{2\pi}{n}$, and two nonzero complex numbers $\alpha=\cos t$, $\beta=\imath \sin t $ such that
$$\exp(\imath t)=\exp(\imath t(1+n))=\exp(\imath t(1-n))=\alpha+\beta,~~ \exp(-\imath t)=\alpha-\beta.$$
By Theorems \ref{maintheorem} and \ref{maintheorem-case2}, $(\cos(\frac{2\pi}{n}), \imath \sin (\frac{2\pi}{n}))$-revival occurs on $\Gamma$ from a vertex $g$ to a vertex $h$ for $n/2=g^{-1}h\in \mathbb{Z}_n$. Since $X\neq \emptyset$, by Corollary \ref{maincor}, $(\alpha, \beta)$-revival does not occur on $\Gamma$ from a vertex $g$ to a vertex $h$ for $g^{-1}h\notin \mathbb{Z}_n$.
}
\end{example}

\end{document}